\documentclass{elsarticle}
\biboptions{numbers,sort&compress}
\usepackage{hyperref,amssymb,amsmath,amsfonts,booktabs,multirow,algorithm}

\newtheorem{theorem}{Theorem}[section]
\newtheorem{lemma}[theorem]{Lemma}

\newtheorem{example}[theorem]{Example}
\newtheorem{corollary}[theorem]{Corollary}
\newcommand{\ind}{\textup{ind}}

\newdefinition{remark}{Remark}
\newproof{proof}{Proof}
\numberwithin{equation}{section}

\journal{*************}

\def\sbmatrix{\left[\begin{array}}
\def\endsbmatrix{\end{array}\right]}

\begin{document}

\begin{frontmatter}

\title{Formulae for the Drazin inverse of Modified Tensors via the Einstein Product\tnoteref{mytitlenote}}
\tnotetext[mytitlenote]{The second author is supported by the National Natural Science Foundation of China (NSFC) (No. 11901079), and China Postdoctoral Science Foundation (No. 2021M700751), and the Scientific and Technological Research Program Foundation of Jilin Province (No. JJKH20190690KJ; No. JJKH20220091KJ; No. JJKH20250851KJ).
The third author is supported by the Ministry of Education, Science and Technological Development, Republic of Serbia (No. 451-03-65/2024-03/200124).}

\author{Yue Zhao}
\ead{yuezhao0303@163.com}
\address{Department of Math \& Stat, Georgia State University, Atlanta, GA 30303, USA.}

\author{Daochang Zhang}
\ead{daochangzhang@126.com}
\address{College of Sciences, Northeast Electric Power University, Jilin, P.R. China.}

\author{Dijana Mosi\' c}
\ead{dijana@pmf.ni.ac.rs}
\address{Faculty of Sciences and Mathematics, University of Ni\v s, P.O.
Box 224, 18000 Ni\v s, Serbia.}

\begin{abstract}
This paper establishes exact expressions for the Drazin inverse of the modified tensor $\mathcal A-\mathcal C*_N\mathcal D^D*_N\mathcal B$ via the Einstein product, formulated using the Drazin inverse of $\mathcal A$ and the generalized Schur complement $\mathcal D-\mathcal B*_N\mathcal A^{D}*_N\mathcal C$, providing a comprehensive generalization and unification of existing results in the literature for the case when the tensors are of order two. Furthermore, the findings reduce to the classical Sherman-Morrison-Woodbury formula in the special case of second-order tensors. Finally, we give an example to illustrate our new explicit expression.
\end{abstract}

\begin{keyword}
Drazin inverse\sep Modified Tensor\sep Einstein Product\sep Sherman-Morrison-Woodbury formula
\MSC[2010] 15A09; 39B42; 53A45; 65F20
\end{keyword}
\end{frontmatter}

\section{Introduction}
Tensors are higher-dimensional extensions of matrices. High-order tensors are extensively applied in analytical mechanics, data mining, neural networks, computer vision, chemometrics, and signal processing \cite{EL200797, BDML199093, Wei(2020), KTBB200999, CP200194, SRCL2012}.

For positive integers $N$ and $I_1,...,I_N$, let $[N]=\{1,...,N\}$ and $\mathbb{C}^{I_1\times \cdot\cdot\cdot \times I_N}$ represents the set of order-$N$ tensors with dimensions $I_1\times \cdot\cdot\cdot \times I_N$ over the complex field \cite{WDYW2016}. A tensor
$\mathcal E=(e_{i_1,...,i_Ni_1,...,i_N})\in\mathbb{C}^{I_1\times \cdot\cdot\cdot \times I_N\times I_1\times \cdot\cdot\cdot \times I_N}$ is called a diagonal tensor if all off-diagonal elements are zero and the diagonal elements are nonzero. When the diagonal elements are equal to $1$, $\mathcal E$ is called the unit tensor, denoted by $\mathcal I$. The zero tensor, denoted by 0, is the tensor where all entries are zero.

The Einstein product of two tensors is defined in \cite{AE2007} by the operation $*_N$ as follows:
\begin{equation*}
(\mathcal A*_N\mathcal B)_{i_1\cdot\cdot\cdot i_Nj_1\cdot\cdot\cdot j_M}=\sum_{h_1\cdot\cdot\cdot h_{N^\prime}}a_{i_1\cdot\cdot\cdot i_Nh_1\cdot\cdot\cdot h_{N^\prime}}b_{h_1\cdot\cdot\cdot h_{N^\prime}j_1\cdot\cdot\cdot j_M},
\end{equation*}
where $\mathcal A\in\mathbb{C}^{I_1\times \cdot\cdot\cdot \times I_N\times H_1\times \cdot\cdot\cdot \times H_{N^\prime}}$ and
$\mathcal B\in\mathbb{C}^{H_1\times \cdot\cdot\cdot \times H_{N^\prime}\times J_1\times \cdot\cdot\cdot \times J_M}$. This tensor product satisfies the associative law and $\mathcal A*_N\mathcal B\in\mathbb{C}^{I_1\times\cdot\cdot\cdot\times I_N\times J_1\times\cdot\cdot\cdot\times J_M}$.

A tensor $\mathcal A\in\mathbb{C}^{I_1\times \cdot\cdot\cdot \times I_N\times I_1\times \cdot\cdot\cdot \times I_N}$ is termed invertible if there exists a tensor $\mathcal F\in\mathbb{C}^{I_1\times \cdot\cdot\cdot \times I_N\times I_1\times \cdot\cdot\cdot \times I_N}$ such that $\mathcal A*_N\mathcal F=\mathcal F*_N\mathcal A=\mathcal I$. In this case, $\mathcal F$ is called the inverse of $\mathcal A$, denoted as $\mathcal A^{-1}=\mathcal F$.
Now, we state the definition of the Drazin inverse of tensors below \cite{JWtensor2018}.

Furthermore, for $\mathcal A\in\mathbb{C}^{I_1\times \cdot\cdot\cdot \times I_N\times J_1\times \cdot\cdot\cdot \times J_M}$, the range and the null space of $\mathcal A$ are denoted by $\mathbb{R}(\mathcal A)$ and $\mathbb{N}(\mathcal A)$, respectively, and are defined as follows:
$$\mathbb{R}(\mathcal A)=\{\mathcal A*_N\mathcal X: \mathcal X\in\mathbb{C}^{J_1\times \cdot\cdot\cdot \times J_M}\},$$
$$\mathbb{N}(\mathcal A)=\{\mathcal X\in\mathbb{C}^{J_1\times \cdot\cdot\cdot \times J_M}: \mathcal A*_N\mathcal X=0\}.$$

Let $\mathcal A\in\mathbb{C}^{I_1\times \cdot\cdot\cdot \times I_N\times I_1\times \cdot\cdot\cdot \times I_N}$. The smallest non-negative integer $l$ such that $\mathbb{R}(\mathcal A^l)=\mathbb{R}(\mathcal A^{l+1})$ is referred to as the index of $\mathcal A$, denoted by $\ind(\mathcal A)$. For a square tensor $\mathcal A$, the $l$-th power of $\mathcal A$ is denoted by $\mathcal A^l$, with $\mathcal A^0=\mathcal I$. In addition, $\mathcal A$ is called $l$-nilpotent when $\mathcal A^l=0$ for positive integer $l$.
A tensor $\mathcal X\in\mathbb{C}^{I_1\times \cdot\cdot\cdot \times I_N\times I_1\times \cdot\cdot\cdot \times I_N}$ that satisfies the following equations:
\begin{eqnarray*}
\mathcal A^l*_N\mathcal X*_N\mathcal A=\mathcal A^l,~\mathcal X*_N\mathcal A*_N\mathcal X=\mathcal X,~\mathcal A*_N\mathcal X=\mathcal X*_N\mathcal A,
\end{eqnarray*}
is defined as the Drazin inverse of $\mathcal A$, denoted by $\mathcal X=\mathcal A^D$, where $l\geq\ind(\mathcal A)$.
When $\ind(\mathcal A)=1$, the Drazin inverse is termed as the group inverse, denoted by $\mathcal A^{\sharp}$. If $\ind(\mathcal A)=0$, then $\mathcal A^D=A^{-1}$. It holds that $(\mathcal A^D)^k=(\mathcal A^k)^D$ for any nonnegative integer $k$. For tensors of order two, the above generalized inverse is the Drazin inverse of matrices \cite{ABTG1974,daochangjcam2025}. Meanwhile, the Einstein product of the block tensors is similar to the product of block matrices (see \cite{LSB2021}). We define the sum to be $0$, if the lower limit of a sum is greater than its upper limit.

When the tensors are of order two, the Sherman-Morrison-Woodbury formula is given as
$$
  (\mathcal A-\mathcal C*_N\mathcal D^{-1}*_N\mathcal B)^{-1}=\mathcal A^{-1}+\mathcal A^{-1}*_N\mathcal C*_N(\mathcal D-\mathcal B*_N\mathcal A^{-1}*_N\mathcal C)^{-1}*_N\mathcal B*_N\mathcal A^{-1},
$$
where $\mathcal A$, $\mathcal D$, and $\mathcal D-\mathcal B*_N\mathcal A^{-1}*_N\mathcal C$ are invertible (and so $\mathcal A-\mathcal C*_N\mathcal D^{-1}*_N\mathcal B$) (see \cite{Sherman(1950),Woodbury(1950)}). The tensor $\mathcal A-\mathcal C*_N\mathcal D^{-1}*_N\mathcal B$ is termed as the modified tensor of $\mathcal A$, when $\mathcal D$ is an unit tensor; while $\mathcal D-\mathcal B*_N\mathcal A^{-1}*_N\mathcal C$ is defined as the Schur complement of $\mathcal A$. A key challenge in this field lies in deriving formulae for the Drazin inverse of a modified tensor in terms of the Drazin inverse of the original tensor. Such formulae are highly relevant to applications in which a tensor is represented as the sum of a structured tensor and a perturbation, including updating problems \cite{Hager Updating(1989)}. These inverse modification formulae have been extensively studied and find applications in structural analysis, statistics, numerical analysis, network theory, and partial differential equations (see \cite{Harte(1988),Horn Johnson(1985)}).

In particular, the generalized inverses of tensors have been extensively studied in the literature. In 2014, Li and Ng \cite{Ng(2014)} studied the perturbation bound for the spectral radius of an $m$th-order $n$-dimensional nonnegative tensor. Two years later, Che et al. \cite{CheWei(2016)} investigated the perturbation bounds of the tensor eigenvalue and singular value problems of even order. In 2018, Ji and Wei \cite{JWtensor2018}  proposed a method for computing the Drazin inverse of even-order tensors using the Einstein product, with applications to solving singular tensor equations. Subsequent studies further explored the general properties of the Drazin inverse of tensors under the Einstein product. For example, in 2021, Liu et al. \cite{LSB2021} derived expressions for the Drazin inverse of the sum of two tensors via the Einstein product.  The extensive body of work on generalized inverses of tensors and their applications \cite{BNS,RB32017,LAJNEF2022,HJMB2017,MLBZ2019,LSBZCB2016,LWD2020JCAM,
WXD2021LAMA} has motivated our investigation into the Drazin inverse of modified tensors.

Throughout this paper, let
\begin{align*}
&\mathcal A\in\mathbb{C}^{I_1\times \cdot\cdot\cdot \times I_N\times I_1\times \cdot\cdot\cdot \times I_N},~\mathcal D\in\mathbb{C}^{H_1\times \cdot\cdot\cdot \times H_M\times H_1\times \cdot\cdot\cdot \times H_M},\\~&\mathcal B\in\mathbb{C}^{H_1\times \cdot\cdot\cdot \times H_M\times I_1\times \cdot\cdot\cdot \times I_N},~\mathcal C\in\mathbb{C}^{I_1\times \cdot\cdot\cdot \times I_N\times H_1\times \cdot\cdot\cdot \times H_M}.&
\end{align*}
For convenience, we adopt the following symbolic notations:
\begin{align*}
&\mathcal S_{\mathcal A^e}=\mathcal A^e*_N\mathcal S*_N\mathcal A^e,~\mathcal Z_{\mathcal D^e}=\mathcal D^e*_N\mathcal Z*_N\mathcal D^e,~\mathcal Y=\mathcal C*_N\mathcal D^D*_N\mathcal B,~\mathcal R=\mathcal B*_N\mathcal A^D*_N\mathcal C,\\&\mathcal H=\mathcal B*_N\mathcal A^{D},~\mathcal K=\mathcal A^{D}*_N\mathcal C,~\mathcal F=\mathcal D^D*_N\mathcal B,~\mathcal E=\mathcal C*_N\mathcal D^D,~\mathcal A^\pi=\mathcal I-\mathcal A*_N\mathcal A^D.&
\end{align*}
For simplicity, we denote the modified tensor $\mathcal A-\mathcal C*_N\mathcal D^{D}*_N\mathcal B$ by $\mathcal S$ and the generalized Schur complement $\mathcal D-\mathcal B*_N\mathcal A^{D}*_N\mathcal C$ by $\mathcal Z$. Utilizing properties of Dedekind finiteness in unit tensor subalgebras, we relax and eliminate certain conditions in  the theorems of  \cite{Daochangf, Dijanaf, daochang29}, providing representations of $(\mathcal A-\mathcal C*_N\mathcal D^{D}*_N\mathcal B)^D$ via the Einstein product under weaker assumptions, generalizing and unifying results in \cite{daochang10, 3, 10, 9} and the Sherman-Morrison-Woodbury formula when the tensors are of order two.

Furthermore, we give a numerical example and compute its perturbation error. Moreover, by comparing with Mathematica's \texttt{DrazinInverse} function, our main result provides a more accurate method for computing the Drazin inverse.

\section{Key lemmas}
We start with the Drazin inverse for the sum of two block tensors. Based on different tensor decompositions, we then derive several formulae for the Drazin inverse of a modified tensor $\mathcal S$, which are extremely important for obtaining our main results.
\begin{lemma}{\rm\cite[Theorem 3.1]{LSB2021}}\label{PQ=0}
Let $\mathcal P\in\mathbb{C}^{I_1\times \cdot\cdot\cdot \times I_N \times I_1\times\cdot\cdot\cdot\times I_N}$ and $\mathcal Q\in\mathbb{C}^{I_1\times \cdot\cdot\cdot \times I_N \times I_1\times\cdot\cdot\cdot\times I_N}$.\\
If $\mathcal P*_N\mathcal Q=0$, then
\begin{eqnarray*}\label{pq=0}
(\mathcal P+\mathcal Q)^D&=&(\mathcal I-\mathcal Q*_N\mathcal Q^D)*_N
(\sum_{i=0}^{k-1}\mathcal Q^i*_N(\mathcal P^D)^i)*_N\mathcal P^D  \\&+&\mathcal Q^D*_N(\sum_{i=0}^{k-1}(\mathcal Q^D)^i*_N\mathcal P^i)*_N(\mathcal I-\mathcal P*_N\mathcal P^D),
\end{eqnarray*}
where $max\{\ind(\mathcal P),\ind(\mathcal Q)\}\leq k\leq \ind(\mathcal P)+\ind(\mathcal Q)$.
\end{lemma}

Here, we present several key lemmas concerning the Drazin inverse of the modified tensor $\mathcal S$.
\begin{lemma}\label{2.2}
If $\mathcal S*_N\mathcal A^\pi$ is t-nilpotent and if

(a)~$\mathcal S*_N\mathcal A^{\pi}*_N\mathcal Y*_N\mathcal A^e=0$, then
\begin{eqnarray*}
\mathcal S^{D}&=&-\sum_{i=0}^{t-2}\left[(\mathcal S_{\mathcal A^e}^{D})^{i+2}
-\mathcal A^{\pi}*_N\mathcal Y*_N\mathcal (S_{\mathcal A^e}^{D})^{i+3}\right]*_N\mathcal Y*_N\mathcal A^\pi *_N\mathcal S^i
\nonumber\\&+&\mathcal S_{\mathcal A^e}^{D}
-\mathcal A^\pi *_N\mathcal Y*_N(\mathcal S_{\mathcal A^e}^{D})^2,
\end{eqnarray*}

(b)~$\mathcal S*_N\mathcal A^{e}*_N\mathcal Y*_N\mathcal A^\pi=0$, then
\begin{eqnarray*}
\mathcal S^{D}&=&-\sum_{i=0}^{t-1}\mathcal S^{i}*_N\mathcal A^\pi
*_N\mathcal Y*_N\mathcal (S_{\mathcal A^e}^{D})^{i+2}
+\mathcal S_{\mathcal A^e}^{D}.
\end{eqnarray*}
\end{lemma}
\begin{proof}
$(a)$~Note that $\mathcal S=\mathcal S*_N\mathcal A^{\pi}+(\mathcal S_{\mathcal A^e}-\mathcal A^\pi *_N\mathcal Y*_N\mathcal A^e):=\mathcal P+\mathcal Q$.
Since $\mathcal S*_N\mathcal A^\pi$ is $t$-nilpotent, we conclude that $(\mathcal S*_N\mathcal A^\pi)^{D}=0$ and $(\mathcal S*_N\mathcal A^\pi)^{\pi}=\mathcal I$. By calculation, we have $(\mathcal S*_N\mathcal A^\pi)^n=\mathcal S*_N\mathcal A^\pi *_N\mathcal S^{n-1},~n\geq 1$.

From $\mathcal S*_N\mathcal A^{\pi}*_N\mathcal Y*_N\mathcal A^e=0$, we have $\mathcal P*_N \mathcal Q=\mathcal S*_N\mathcal A^\pi*_N(\mathcal S_{\mathcal A^e}-\mathcal A^\pi *_N\mathcal Y*_N\mathcal A^e)=0$. Furthermore, by Lemma \ref{PQ=0}, we obtain
\begin{equation}\label{sd13}
\mathcal S^{D}=\sum_{i=0}^{t-1}\left[(\mathcal S_{\mathcal A^e}-\mathcal A^\pi *_N\mathcal Y *_N\mathcal A^e)^{D}\right]^{i+1}*_N(\mathcal S*_N\mathcal A^\pi)^i.
\end{equation}
Also, notice that $\mathcal S_{\mathcal A^e}*_N(\mathcal A^\pi *_N\mathcal Y*_N\mathcal A^e)=0$, and $(\mathcal A^\pi *_N\mathcal Y*_N\mathcal A^e)^2=0$, which means $\mathcal A^\pi *_N\mathcal Y*_N\mathcal A^e$ is 2-nilpotent, and thus $(\mathcal A^\pi *_N\mathcal Y*_N\mathcal A^e)^D=0$ and $(\mathcal A^\pi *_N\mathcal Y*_N\mathcal A^e)^\pi=\mathcal I$. Due to Lemma \ref{PQ=0}, we derive
$$(\mathcal S_{\mathcal A^e}-\mathcal A^\pi *_N\mathcal Y*_N\mathcal A^e)^D=\mathcal S_{\mathcal A^e}^{D}
-\mathcal A^\pi *_N\mathcal Y*_N(\mathcal S_{\mathcal A^e}^{D})^{2}.$$
As for the formula of $\mathcal S_{\mathcal A^e}^{D}
-\mathcal A^\pi *_N\mathcal Y*_N(\mathcal S_{\mathcal A^e}^{D})^{2}$ to the power of $n$ $(n\geq 1)$, it follows that
$$(\mathcal S_{\mathcal A^e}^{D}
-\mathcal A^\pi *_N\mathcal Y*_N(\mathcal S_{\mathcal A^e}^{D})^{2})^n=
(\mathcal S_{\mathcal A^e}^{D})^{n}
  -\mathcal A^\pi *_N\mathcal Y*_N(\mathcal S_{\mathcal A^e}^{D})^{n+1}.$$
Substituting the above results into \eqref{sd13}, we get the formula below:
\begin{eqnarray*}
\mathcal S^{D}&=&\sum_{i=0}^{t-2}\left[(\mathcal S_{\mathcal A^e}^{D})^{i+2}
-\mathcal A^{\pi}*_N\mathcal Y*_N\mathcal (S_{\mathcal A^e}^{D})^{i+3}\right]*_N\mathcal S*_N\mathcal A^\pi *_N\mathcal S^i
\nonumber\\&+&\mathcal S_{\mathcal A^e}^{D}
-\mathcal A^\pi *_N\mathcal Y*_N(\mathcal S_{\mathcal A^e}^{D})^2.
\end{eqnarray*}
Hence, after simplifying the above formula, we complete the proof.

$(b)$~By leveraging the structure of the proof of part $(a)$, a similar result can be readily derived. The main steps are outlined as follows:
$$\mathcal S=\mathcal S_{\mathcal A^e}+\mathcal A^\pi *_N\mathcal S*_N\mathcal A^e+\mathcal S*_N\mathcal A^{\pi} ~{\rm and}~
(\mathcal S*_N\mathcal A^{\pi})^n=\mathcal S^{n}*_N\mathcal A^{\pi},~n\geq1.$$
Then, we derive the representation of $\mathcal S^{D}$ below:
\begin{eqnarray}\label{simplify1}
\mathcal S^{D}&=&\sum_{i=0}^{t-2}\mathcal S^{i+1}*_N\mathcal A^\pi
*_N\left[(\mathcal S_{\mathcal A^e}^{D})^{i+2}
-\mathcal A^{\pi}*_N\mathcal Y*_N\mathcal (S_{\mathcal A^e}^{D})^{i+3}\right]
\nonumber\\&+&\mathcal S_{\mathcal A^e}^{D}
-\mathcal A^\pi *_N\mathcal Y*_N(\mathcal S_{\mathcal A^e}^{D})^2,
\end{eqnarray}
After simplifying \eqref{simplify1}, we finish the proof.
\end{proof}

\begin{lemma}\label{sd24}
If $\mathcal A^\pi*_N\mathcal S$ is $r$-nilpotent and if

(a)~$\mathcal A^e*_N\mathcal Y*_N\mathcal A^{\pi}*_N\mathcal S=0$, then
\begin{eqnarray*}
\mathcal S^{D}&=&\sum_{i=0}^{r-2}\mathcal S^i*_N\mathcal A^\pi*_N\mathcal S*_N\left[(\mathcal S_{\mathcal A^e}^{D})^{i+2}
-(\mathcal S_{\mathcal A^e}^{D})^{i+3}*_N\mathcal Y*_N\mathcal A^\pi\right]
\nonumber\\&+&\mathcal S_{\mathcal A^e}^{D}
-(\mathcal S_{\mathcal A^e}^{D})^2*_N\mathcal Y*_N\mathcal A^\pi,
\end{eqnarray*}

(b)~$\mathcal A^\pi*_N\mathcal Y*_N\mathcal A^{e}*_N\mathcal S=0$, then
\begin{eqnarray*}
\mathcal S^{D}&=&-\sum_{i=0}^{r-1}
(\mathcal S_{\mathcal A^e}^{D})^{i+2}*_N\mathcal Y*_N\mathcal A^\pi*_N\mathcal S^{i}+\mathcal S_{\mathcal A^e}^{D},
\end{eqnarray*}
\end{lemma}
\begin{proof}
$(a)$~We observe that $\mathcal S=(\mathcal S_{\mathcal A^e}-\mathcal A^e*_N\mathcal Y*_N\mathcal A^\pi)+\mathcal A^\pi *_N\mathcal S:=\mathcal P+\mathcal Q$.
Owing to $\mathcal A^\pi*_N\mathcal S$ is r-nilpotent, we see that $(\mathcal A^\pi*_N\mathcal S)^{D}=0$ and $(\mathcal A^\pi*_N\mathcal S)^{\pi}=\mathcal I$. By calculation, we obtain $(\mathcal A^\pi*_N\mathcal S)^n=\mathcal S^{n-1}*_N\mathcal A^\pi *_N\mathcal S,~n\geq1$.

From $\mathcal A^e*_N\mathcal Y*_N\mathcal A^{\pi}*_N\mathcal S=0$, we have $\mathcal P*_N\mathcal Q=(\mathcal S_{\mathcal A^e}-\mathcal A^e *_N\mathcal Y*_N\mathcal A^\pi)*_N\mathcal A^\pi*_N\mathcal S=0$. Applying Lemma \ref{PQ=0}, we get
\begin{equation}\label{sd}
\mathcal S^{D}=\sum_{i=0}^{r-1}(\mathcal A^\pi *_N\mathcal S)^i*_N\left[(\mathcal S_{\mathcal A^e}-\mathcal A^e *_N\mathcal Y*_N\mathcal A^\pi)^{D}\right]^{i+1}.
\end{equation}
Also, note that $\mathcal A^e*_N\mathcal Y*_N\mathcal A^\pi*_N\mathcal S_{\mathcal A^e}=0$ satisfies the condition of Lemma \ref{PQ=0}. Additionally, $(\mathcal A^e *_N\mathcal S*_N\mathcal A^\pi)^2=0$, which means $\mathcal A^e *_N\mathcal S*_N\mathcal A^\pi$ is 2-nilpotent. Consequently, $(\mathcal A^e *_N\mathcal S*_N\mathcal A^\pi)^D=0$ and $(\mathcal A^e *_N\mathcal S*_N\mathcal A^\pi)^\pi=\mathcal I$. Therefore, we derive
$$(\mathcal S_{\mathcal A^e}-\mathcal A^e *_N\mathcal Y*_N\mathcal A^\pi)^D=\mathcal S_{\mathcal A^e}^{D}
-(\mathcal S_{\mathcal A^e}^{D})^{2}*_N\mathcal Y*_N\mathcal A^\pi.$$
Hence, the formula of $\mathcal S_{\mathcal A^e}^{D}
-(\mathcal S_{\mathcal A^e}^{D})^{2}*_N\mathcal Y*_N\mathcal A^\pi$ to the power of $n$ $(n\geq 1)$ is represented as follows:
$$(\mathcal S_{\mathcal A^e}^{D}
-(\mathcal S_{\mathcal A^e}^{D})^{2}*_N\mathcal Y*_N\mathcal A^\pi)^n=
(\mathcal S_{\mathcal A^e}^{D})^{n}
  - (\mathcal S_{\mathcal A^e}^{D})^{n+1}*_N\mathcal Y*_N\mathcal A^\pi.$$
After substituting the above results into \eqref{sd}, we have
\begin{eqnarray*}
\mathcal S^{D}&=&\sum_{i=0}^{r-2}\mathcal S^i*_N\mathcal A^\pi*_N\mathcal S*_N\left[(\mathcal S_{\mathcal A^e}^{D})^{i+2}
-(\mathcal S_{\mathcal A^e}^{D})^{i+3}*_N\mathcal Y*_N\mathcal A^\pi\right]
\nonumber\\&+&\mathcal S_{\mathcal A^e}^{D}
-(\mathcal S_{\mathcal A^e}^{D})^2*_N\mathcal Y*_N\mathcal A^\pi,
\end{eqnarray*}
Therefore, after adjusting the upper and lower limits of the sum and simplifying the formula, we finish the proof.

$(b)$~Taking the same reasoning approach as in the proof of $(a)$, a similar conclusion follows naturally. The principal results are provided below:
$$\mathcal S=\mathcal A^\pi *_N\mathcal S-\mathcal A^e*_N\mathcal Y*_N\mathcal A^\pi+\mathcal S_{\mathcal A^e} ~{\rm and}~
(\mathcal A^{\pi}*_N\mathcal S)^n=\mathcal A^{\pi}*_N\mathcal S^{n},~n\geq1.$$
Furthermore, we obtain the following representation
\begin{eqnarray*}
\mathcal S^{D}&=&\sum_{i=0}^{r-2}
\left[(\mathcal S_{\mathcal A^e}^{D})^{i+2}
-(\mathcal S_{\mathcal A^e}^{D})^{i+3}*_N\mathcal Y*_N\mathcal A^\pi\right]*_N\mathcal A^\pi*_N\mathcal S^{i+1}
\nonumber\\&+&\mathcal S_{\mathcal A^e}^{D}
-(\mathcal S_{\mathcal A^e}^{D})^2*_N\mathcal Y*_N\mathcal A^\pi.
\end{eqnarray*}
Thus, the proof is concluded upon refining the summation bounds and simplifying the expression.
\end{proof}

In what follows, we refine and extend  \cite[Lemma 2.2 $(a)(c)$]{Dijanaf} by removing one of its assumptions and proposing the following corollaries.
\begin{corollary}\label{1.1}
If $\mathcal S*_N\mathcal A^\pi*_N\mathcal Y=0$ and $\ind(\mathcal A)=k$, then
\begin{eqnarray*}
\mathcal S^{D}&=&-\sum_{i=0}^{k-1}\left[(\mathcal S_{\mathcal A^e}^{D})^{i+2}
-\mathcal A^{\pi}*_N\mathcal Y*_N\mathcal (S_{\mathcal A^e}^{D})^{i+3}\right]*_N\mathcal Y*_N\mathcal A^\pi *_N\mathcal A^i
\nonumber\\&+&\mathcal S_{\mathcal A^e}^{D}
-\mathcal A^\pi *_N\mathcal Y*_N(\mathcal S_{\mathcal A^e}^{D})^2.
\end{eqnarray*}
\end{corollary}
\begin{proof}
The proof proceeds similarly to Lemma \ref{2.2}.
\end{proof}

\begin{corollary}\label{2.1}
If $\mathcal Y*_N\mathcal A^\pi*_N\mathcal S=0$ and $\ind(\mathcal A)=k$, then
\begin{eqnarray*}
\mathcal S^{D}&=&-\sum_{i=0}^{k-1}\mathcal A^i*_N\mathcal A^\pi*_N\mathcal Y*_N\left[(\mathcal S_{\mathcal A^e}^{D})^{i+2}
-(\mathcal S_{\mathcal A^e}^{D})^{i+3}*_N\mathcal Y*_N\mathcal A^\pi\right]
\nonumber\\&+&\mathcal S_{\mathcal A^e}^{D}
-(\mathcal S_{\mathcal A^e}^{D})^2*_N\mathcal Y*_N\mathcal A^\pi.
\end{eqnarray*}
\end{corollary}
\begin{proof}
The proof follows the same approach as Lemma \ref{sd24}.
\end{proof}

\begin{lemma}\label{ece}
  Let $\mathcal X,~\mathcal Y,~\mathcal I\in\mathbb{C}^{I_1\times \cdot\cdot\cdot \times I_N \times I_1\times\cdot\cdot\cdot\times I_N}$ and $\mathcal I^{2}=\mathcal I$. If $\mathcal I*_N\mathcal X*_N\mathcal I*_N\mathcal Y*_N\mathcal I=\mathcal I$, then $\mathcal I*_N\mathcal Y*_N\mathcal I*_N\mathcal X*_N\mathcal I=\mathcal I$.
\end{lemma}
\begin{proof}
Let $\mathcal W=\{\mathcal M\in\mathbb{C}^{I_1\times \cdot\cdot\cdot \times I_N \times I_1\times\cdot\cdot\cdot\times I_N}~~|~~\mathcal I*_N\mathcal M=\mathcal M*_N\mathcal I=\mathcal M\}$. Then $\mathcal W$ is a finite dimensional tensor algebra over $\mathbb{C}$ with unit tensor $\mathcal I$, and thus $\mathcal W$ is Dedekind finite. Note that $\mathcal I*_N\mathcal X*_N\mathcal I,~\mathcal I*_N\mathcal Y*_N\mathcal I\in\mathcal W$ and $(\mathcal I*_N\mathcal X*_N\mathcal I)*_N(\mathcal I*_N\mathcal Y*_N\mathcal I)=\mathcal I$. Therefore, $(\mathcal I*_N\mathcal Y*_N\mathcal I)*_N(\mathcal I*_N\mathcal X*_N\mathcal I)=\mathcal I$, that is, $\mathcal I*_N\mathcal Y*_N\mathcal I*_N\mathcal X*_N\mathcal I=\mathcal I$.
\end{proof}

Let $\mathcal H=\mathcal B*_N\mathcal A^{D}$ and $\mathcal K=\mathcal A^{D}*_N\mathcal C$.
\begin{lemma}\label{2.7sae}
 Let $\mathcal T=\mathcal A^{D}+\mathcal K*_N\mathcal Z^{D}*_N\mathcal H$. Then the following statements are equivalent:
 \begin{enumerate}
   \item[(1)] $\mathcal K*_N\mathcal D^{\pi}*_N\mathcal Z^{D}*_N\mathcal H
   =\mathcal K*_N\mathcal D^{D}*_N\mathcal Z^{\pi}*_N\mathcal H$;
   \item[(2)] $\mathcal S_{\mathcal A^e}*_N\mathcal T
   =\mathcal A*_N\mathcal A^{D}$;
   \item[(3)] $\mathcal T*_N\mathcal S_{\mathcal A^e}=\mathcal A*_N\mathcal A^{D}$;
   \item[(4)] $\mathcal K*_N\mathcal Z^{\pi}*_N\mathcal D^{D}*_N\mathcal H
       =\mathcal K*_N\mathcal Z^{D}*_N\mathcal D^{\pi}*_N\mathcal H$.
\end{enumerate}
Furthermore, if one of (1)--(4) holds, then $\mathcal S_{\mathcal A^e}$ has the group inverse
$$
  \mathcal S_{\mathcal A^e}^{\#}=\mathcal A^{D}+\mathcal K*_N\mathcal Z^{D}*_N\mathcal H.
$$
\end{lemma}
\begin{proof}
Let $\mathcal A^e=\mathcal A*_N\mathcal A^D$ and $\mathcal Z^e=\mathcal Z*_N\mathcal Z^D$. Then
$$\mathcal S_{\mathcal A^e}*_N\mathcal T=\mathcal A^e+\mathcal A*_N\mathcal K*_N\mathcal Z^D*_N\mathcal H-\mathcal A*_N\mathcal K*_N\mathcal D^D*_N\mathcal H-\mathcal A*_N\mathcal K*_N\mathcal D^D*_N(\mathcal D-\mathcal Z)*_N\mathcal Z^D*_N\mathcal H.$$
Hence (2) holds if and only if $\mathcal A*_N\mathcal K*_N(\mathcal Z^D-\mathcal D^D-\mathcal D^D*_N(\mathcal D-\mathcal Z)*_N\mathcal Z^D)*_N\mathcal H=0$, or equivalently $\mathcal K*_N(\mathcal D^\pi*_N\mathcal Z^D-\mathcal D^D*_N\mathcal Z^\pi)*_N\mathcal H=0$, that is (1) holds. Similarly, (3) is equivalent to (4). Lemma \ref{ece} implies equivalence of (2) and (3). Moreover, (2) and (3) give
$$
\mathcal S_{A^e}^{\sharp}=\mathcal T=\mathcal A^D+\mathcal K*_N\mathcal Z^D*_N\mathcal H.
$$
\end{proof}

\section{Main results}

In this section, we present the Drazin inverse of the modified tensor $\mathcal S$ expressed in terms of the Drazin inverses of $\mathcal A$ and the generalized Schur complement $\mathcal D-\mathcal B*_N\mathcal A^{D}*_N\mathcal C$ under different restrictions.

\begin{theorem}\label{Thm3.1}
If $\mathcal S_{\mathcal A^e}*_N\mathcal T=\mathcal A*_N\mathcal A^{D}$, $\mathcal S*_N\mathcal A^\pi$ is t-nilpotent and if

(a)~$\mathcal S*_N\mathcal A^\pi*_N\mathcal Y*_N\mathcal A^e=0$, then
\begin{eqnarray}\label{3.111}
\mathcal S^{D}&=&-\sum_{i=0}^{t-2}(\mathcal I-\mathcal A^\pi*_N\mathcal Y*_N\mathcal T)*_N\mathcal T^{i+2}*_N\mathcal Y*_N\mathcal A^\pi*_N\mathcal S^i
\nonumber\\&+&\mathcal T-\mathcal A^\pi *_N\mathcal Y*_N\mathcal T^2.
\end{eqnarray}

(b)~$\mathcal S*_N\mathcal A^e*_N\mathcal Y*_N\mathcal A^\pi=0$, then
\begin{eqnarray*}
\mathcal S^{D}&=&-\sum_{i=0}^{t-1}\mathcal S^{i}*_N\mathcal A^\pi*_N\mathcal Y*_N\mathcal T^{i+2}+\mathcal T.
\end{eqnarray*}
\end{theorem}
\begin{proof}
By Lemma \ref{2.2} and Lemma \ref{2.7sae}, the proof follows directly.
\end{proof}

Note that Theorem \ref{Thm3.1} $(a)$ extends \cite[Theorem 2.5 $(a)$]{Dijanaf} and \cite[Theorem 2.5]{Daochangf}, while $(b)$ generalizes \cite[Theorem 2.5 $(b)$]{Dijanaf} and \cite[Theorem 2.9]{Daochangf}. Also, the above theorem can generalize and unify some results, which are listed partially as follows:
\begin{enumerate}
  \item[(1)] $\mathcal A^{\pi}*_N\mathcal C=0,~\mathcal B*_N\mathcal A^{\pi}=0,
      ~\mathcal C*_N\mathcal D^{\pi}*_N\mathcal Z^{d}*_N\mathcal B=0,~\mathcal C*_N\mathcal D^{d}*_N\mathcal Z^{\pi}*_N\mathcal B=0,~\mathcal C*_N\mathcal Z^{d}*_N\mathcal D^{\pi}*_N\mathcal B=0$
          and $\mathcal C*_N\mathcal Z^{\pi}*_N\mathcal D^{d}*_N\mathcal B=0$ (see \cite[Theorem 2.1]{daochang10});
  \item[(2)] $\mathcal A^{\pi}*_N\mathcal C=0, ~\mathcal C*_N\mathcal D^{\pi}=0, ~\mathcal Z^{\pi}*_N\mathcal B=0,~\mathcal D^{\pi}*_N\mathcal B=0$ and $\mathcal C*_N\mathcal Z^{\pi}=0$ (see \cite[Theorem 2.1]{3});

  \item[(3)] $\mathcal A^{\pi}*_N\mathcal C=0,~\mathcal C*_N\mathcal D^{\pi}*_N\mathcal Z^{d}*_N\mathcal B=0,~\mathcal C*_N\mathcal D^{d}*_N\mathcal Z^{\pi}*_N\mathcal B=0,~\mathcal C*_N\mathcal Z^{d}*_N\mathcal D^{\pi}*_N\mathcal B=0$ and $\mathcal C*_N\mathcal Z^{\pi}*_N\mathcal D^{d}*_N\mathcal B=0$ (see \cite[Theorem 2.1]{10});
  \item[(4)] $\mathcal A^{\pi}*_N\mathcal C=\mathcal C*_N\mathcal D^{\pi},~\mathcal D^{\pi}*_N\mathcal B=0$ and $\mathcal D*_N\mathcal Z^{\pi}=0$ (see \cite[Theorem 2.1]{9});
  \item[(5)] $\mathcal A^{\pi}*_N\mathcal C=\mathcal C*_N\mathcal D^{\pi},~\mathcal D^{\pi}*_N\mathcal B=0,~\mathcal Z^{\pi}*_N\mathcal B=0$ and $\mathcal C*_N\mathcal Z^{\pi}=0$ (see \cite[Theorem 2.2]{9}).
\end{enumerate}

Now, we give another two different expressions of $\mathcal S^D$ in \eqref{3.111}.
\begin{remark}\label{remarka}
From \eqref{3.111}, we observe that
\begin{eqnarray}\label{remark111}
\mathcal T*_N\mathcal S*_N\mathcal A^\pi&=&(\mathcal A^D+\mathcal K*_N\mathcal Z^D*_N\mathcal H)*_N(\mathcal A-\mathcal C*_N\mathcal D^D*_N\mathcal B)*_N\mathcal A^\pi
\nonumber\\&=&(\mathcal A^D*_N\mathcal A-\mathcal K*_N\mathcal D^D*_N\mathcal B-\mathcal K*_N\mathcal Z^D*_N(\mathcal D-\mathcal Z)*_N\mathcal D^D*_N\mathcal B\nonumber\\&+&\mathcal K*_N\mathcal Z^D*_N\mathcal B*_N\mathcal A^D*_N\mathcal A)*_N\mathcal A^\pi
\nonumber\\&=&-\mathcal K*_N\mathcal Z^\pi*_N\mathcal D^D*_N\mathcal B*_N\mathcal A^\pi-\mathcal K*_N\mathcal Z^D*_N(\mathcal I-\mathcal D^\pi)*_N\mathcal B*_N\mathcal A^\pi
\nonumber\\&=&\mathcal K*_N(\mathcal Z^D*_N\mathcal D^\pi-\mathcal Z^\pi*_N\mathcal D^D)*_N\mathcal B*_N\mathcal A^\pi-\mathcal K*_N\mathcal Z^D*_N\mathcal B*_N\mathcal A^\pi.
\end{eqnarray}
From the condition $\mathcal S_{\mathcal A^e}*_N\mathcal T=\mathcal A*_N\mathcal A^{D}$ given in the Theorem \ref{Thm3.1}, by Lemma \ref{2.7sae}, we obtain
$$
\mathcal K*_N(\mathcal Z^D*_N\mathcal D^\pi-\mathcal Z^\pi*_N\mathcal D^D)*_N\mathcal B*_N\mathcal A^\pi=\mathcal K*_N(\mathcal Z^D*_N\mathcal D^\pi-\mathcal Z^\pi*_N\mathcal D^D)*_N\mathcal B.
$$
Thus \eqref{3.111} can be rewritten as
\begin{eqnarray}\label{remark1}
\mathcal S^{D}&=&\sum_{i=0}^{t-2}(\mathcal I-\mathcal A^\pi*_N\mathcal Y*_N\mathcal T)*_N\mathcal T^{i+1}*_N\mathcal K*_N(\mathcal Z^D*_N\mathcal D^\pi-\mathcal Z^\pi*_N\mathcal D^D)*_N\mathcal B*_N\mathcal S^i
\nonumber\\&-&\sum_{i=0}^{t-2}(\mathcal I-\mathcal A^\pi*_N\mathcal Y*_N\mathcal T)*_N\mathcal T^{i+1}*_N\mathcal K*_N\mathcal Z^D*_N\mathcal B*_N\mathcal A^\pi*_N\mathcal S^i
\nonumber\\&+&\mathcal T-\mathcal A^\pi *_N\mathcal Y*_N\mathcal T^2.
\end{eqnarray}
Moreover, note that \eqref{remark111} can be expressed equivalently as
\begin{eqnarray*}\label{tsa}
\mathcal T*_N\mathcal S*_N\mathcal A^\pi&=&\mathcal K*_N(\mathcal Z^D*_N\mathcal D^\pi-\mathcal Z^\pi*_N\mathcal D^D)*_N\mathcal B*_N\mathcal A^\pi-\mathcal K*_N\mathcal Z^D*_N\mathcal B*_N\mathcal A^\pi
\\&=&-\mathcal K*_N(\mathcal Z^D*_N\mathcal D+\mathcal Z^\pi)*_N\mathcal D^D*_N\mathcal B*_N\mathcal A^\pi.
\end{eqnarray*}
Furthermore, we can obtain an alternative equivalent representation of \eqref{remark1}:
\begin{eqnarray}\label{remark222}
\mathcal S^{D}&=&-\sum_{i=0}^{t-2}(\mathcal I-\mathcal A^\pi*_N\mathcal Y*_N\mathcal T)*_N\mathcal T^{i+1}*_N\mathcal K*_N(\mathcal Z^D*_N\mathcal D+\mathcal Z^\pi)*_N\mathcal D^D*_N\mathcal B*_N\mathcal A^\pi*_N\mathcal S^i
\nonumber\\&+&\mathcal T-\mathcal A^\pi *_N\mathcal Y*_N\mathcal T^2.
\end{eqnarray}
\end{remark}

Note that Lemma \ref{2.7sae} implies $\mathcal K*_N\mathcal D^{\pi}*_N\mathcal Z^{D}*_N\mathcal H
   =\mathcal K*_N\mathcal D^{D}*_N\mathcal Z^{\pi}*_N\mathcal H$ and $\mathcal S_{\mathcal A^e}*_N\mathcal T
   =\mathcal A*_N\mathcal A^{D}$ are equivalent. Strengthening $\mathcal K*_N\mathcal D^{\pi}*_N\mathcal Z^{D}*_N\mathcal H
   =\mathcal K*_N\mathcal D^{D}*_N\mathcal Z^{\pi}*_N\mathcal H$ reveals that Corollary \ref{cor3.4} $(a)$ generalizes  \cite[Corollary 2.7]{Daochangf}, while equation $(b)$ is the dual of $(a)$.
\begin{corollary}\label{cor3.4}
If $\mathcal C*_N\mathcal D^\pi*_N\mathcal Z^D*_N\mathcal B=0,~\mathcal C*_N\mathcal D^D*_N\mathcal Z^\pi*_N\mathcal B=0$, $\mathcal S*_N\mathcal A^\pi$ is t-nilpotent and if

(a)~$\mathcal S*_N\mathcal A^\pi*_N\mathcal Y*_N\mathcal A^e=0$, then
\begin{eqnarray*}
\mathcal S^{D}&=&-\sum_{i=0}^{t-2}(\mathcal I-\mathcal A^\pi*_N\mathcal Y*_N\mathcal T)*_N\mathcal T^{i+2}*_N\mathcal Y*_N\mathcal A^\pi*_N\mathcal S^i
\nonumber\\&+&\mathcal T-\mathcal A^\pi *_N\mathcal Y*_N\mathcal T^2.
\end{eqnarray*}

(b)~$\mathcal S*_N\mathcal A^e*_N\mathcal Y*_N\mathcal A^\pi=0$, then
\begin{eqnarray*}
\mathcal S^{D}&=&-\sum_{i=0}^{t-1}\mathcal S^{i}*_N\mathcal A^\pi*_N\mathcal Y*_N\mathcal T^{i+2}+\mathcal T.
\end{eqnarray*}
\end{corollary}

Corollary \ref{cor3.5} directly follows from Theorem \ref{Thm3.1} or Corollary \ref{cor3.4}, wherein $(a)$ recovering \cite[Corollary 2.7]{Daochangf}. Furthermore, the second equation in $(a)$ is derived using Remark \ref{remarka}.
\begin{corollary}\label{cor3.5}
If $\mathcal C*_N\mathcal D^\pi*_N\mathcal Z^D*_N\mathcal B=0,~\mathcal C*_N\mathcal D^D*_N\mathcal Z^\pi*_N\mathcal B=0$ and if

(a)~$\mathcal A^\pi*_N\mathcal Y=0$, then
\begin{eqnarray*}
\mathcal S^{D}&=&-\sum_{i=0}^{k-1}\mathcal T^{i+2}*_N\mathcal Y*_N\mathcal A^\pi*_N\mathcal A^i
+\mathcal T,
\end{eqnarray*} or alternatively
\begin{eqnarray*}
\mathcal S^{D}&=&\sum_{i=0}^{k-1}\mathcal T^{i+1}*_N\mathcal K*_N(\mathcal Z^D*_N\mathcal D^\pi-\mathcal Z^\pi*_N\mathcal D^D)*_N\mathcal B*_N\mathcal A^i
\nonumber\\&-&\sum_{i=0}^{k-1}\mathcal T^{i+1}*_N\mathcal K*_N\mathcal Z^D*_N\mathcal B*_N\mathcal A^\pi*_N\mathcal A^i
+\mathcal T,
\end{eqnarray*}
where $\ind(\mathcal A)=k$.

(b)~$\mathcal Y*_N\mathcal A^\pi=0$, then
\begin{eqnarray*}
\mathcal S^{D}&=&\mathcal T-\sum_{i=0}^{k-1}\mathcal A^{i}*_N\mathcal A^\pi*_N\mathcal Y*_N\mathcal T^{i+2},
\end{eqnarray*}
where $\ind(\mathcal A)=k$.
\end{corollary}

In order to simplify \eqref{remark1}, we strengthen the condition of Corollary \ref{cor3.4} and assume that $\mathcal D^\pi=\mathcal Z^\pi$ to obtain the following corollary, which generalizes \cite[Corollary 2.8]{Daochangf} and  \cite[Corollary 3.2 $(a)$]{Dijanaf}.
\begin{corollary}
If $\mathcal S*_N\mathcal A^\pi*_N\mathcal Y*_N\mathcal A^e=0$, $\mathcal S*_N\mathcal A^\pi$ is t-nilpotent and $\mathcal D^\pi=\mathcal Z^\pi$, then
\begin{eqnarray*}
\mathcal S^{D}&=&-\sum_{i=0}^{t-2}(\mathcal I-\mathcal A^\pi*_N\mathcal Y*_N\mathcal T)*_N\mathcal T^{i+1}*_N\mathcal K*_N\mathcal Z^D*_N\mathcal B*_N\mathcal A^\pi*_N\mathcal S^i
\nonumber\\&+&\mathcal T-\mathcal A^\pi *_N\mathcal Y*_N\mathcal T^2.
\end{eqnarray*}
\end{corollary}

On the other hand, let $\mathcal K*_N(\mathcal Z^D*_N\mathcal D+\mathcal Z^\pi)*_N\mathcal D^D*_N\mathcal B*_N\mathcal A^\pi=0$ given in \eqref{tsa}, we observe that
$
\mathcal K*_N\mathcal Z^D*_N\mathcal B*_N\mathcal A^\pi=\mathcal K*_N(\mathcal Z^D*_N\mathcal D^\pi-\mathcal Z^\pi*_N\mathcal D^D)*_N\mathcal B*_N\mathcal A^\pi.
$
Then, we continue to strengthen the condition to obtain $\mathcal Z^D=\mathcal Z^D*_N\mathcal D^\pi-\mathcal Z^\pi*_N\mathcal D^D$. According to the representation of $\mathcal S^D$ (see \eqref{remark222}), we derive the following corollary.
\begin{corollary}
If $\mathcal S*_N\mathcal A^\pi*_N\mathcal Y*_N\mathcal A^e=0$, $\mathcal S*_N\mathcal A^\pi$ is t-nilpotent and $\mathcal Z^D=\mathcal Z^D*_N\mathcal D^\pi-\mathcal Z^\pi*_N\mathcal D^D$, then
\begin{eqnarray*}
\mathcal S^{D}&=&\mathcal T-\mathcal A^\pi *_N\mathcal Y*_N\mathcal T^2.
\end{eqnarray*}
\end{corollary}

We now consider several specific cases of the preceding Theorem \ref{Thm3.1}, which not only generalize but also unify certain established results in the literature.
Firstly, after strengthening $\mathcal S*_N\mathcal A^\pi*_N\mathcal Y*_N\mathcal A^e=0$ in Theorem \ref{Thm3.1}, we can observe that the following equation $(a)$ recovers \cite[Theorem 2.5]{Daochangf}, while $(b)$ generalizes \cite[Theorem 2.9]{Daochangf}.
\begin{corollary}\label{cor3.3}
If $\mathcal S_{\mathcal A^e}*_N\mathcal T=\mathcal A*_N\mathcal A^{D}$ and if

(a)~$\mathcal A^\pi*_N\mathcal Y=0$, then
\begin{eqnarray*}
\mathcal S^{D}&=&\mathcal T-\sum_{i=0}^{k-1}\mathcal T^{i+2}*_N\mathcal Y*_N\mathcal A^\pi*_N\mathcal A^i
\nonumber,
\end{eqnarray*}

(b)~$\mathcal Y*_N\mathcal A^\pi=0$, then
\begin{eqnarray*}
\mathcal S^{D}&=&\mathcal T-\sum_{i=0}^{k-1}\mathcal A^{i}*_N\mathcal A^\pi*_N\mathcal Y*_N\mathcal T^{i+2},
\end{eqnarray*}
where $\ind(\mathcal A)=k$.
\end{corollary}

We now state another analogous result.
\begin{theorem}\label{Thm3.2}
If $\mathcal S_{\mathcal A^e}*_N\mathcal T=\mathcal A*_N\mathcal A^{D}$, $\mathcal A^\pi*_N\mathcal S$ is r-nilpotent and if

(a)~$\mathcal A^e*_N\mathcal Y*_N\mathcal A^\pi*_N\mathcal S=0$, then
\begin{eqnarray*}
\mathcal S^{D}&=&-\sum_{i=0}^{r-2}\mathcal S^i*_N\mathcal A^\pi*_N\mathcal Y*_N\mathcal T^{i+2}*_N(\mathcal I-\mathcal T*_N\mathcal Y*_N\mathcal A^\pi)
\nonumber\\&+&\mathcal T-\mathcal T^2*_N\mathcal Y*_N\mathcal A^\pi.
\end{eqnarray*}

(b)~$\mathcal A^\pi*_N\mathcal Y*_N\mathcal A^e*_N\mathcal S=0$, then
\begin{eqnarray*}
\mathcal S^{D}&=&-\sum_{i=0}^{r-1}\mathcal T^{i+2}*_N*_N\mathcal Y*_N\mathcal A^\pi*_N\mathcal S^{i}+\mathcal T.
\end{eqnarray*}
\end{theorem}
\begin{proof}
The results follow from Lemma \ref{sd24} and Lemma \ref{2.7sae}.
\end{proof}

It can be seen that Theorem \ref{Thm3.2} $(a)$ generalizes \cite[Theorem 2.5 $(c)$]{Dijanaf} and \cite[Theorem 2.9]{Daochangf}, while $(b)$ extends \cite[Theorem 2.5 $(d)$]{Dijanaf} and \cite[Theorem 2.5]{Daochangf}. Now, we refine and extend  \cite[Theorem 2.5 $(a)$$(c)$]{Dijanaf} by removing one of its assumptions and proposing the following theorem.

\begin{theorem}\label{Thm3.3}
If $\mathcal K*_N\mathcal D^{\pi}*_N\mathcal Z^{D}*_N\mathcal H
   =\mathcal K*_N\mathcal D^{D}*_N\mathcal Z^{\pi}*_N\mathcal H$,
   $ind(\mathcal A)=k$ and if

(a)~$\mathcal S*_N\mathcal A^\pi*_N\mathcal Y=0$, then
\begin{eqnarray*}
\mathcal S^{D}&=&-\sum_{i=0}^{k-1}(\mathcal I-\mathcal A^\pi*_N\mathcal Y*_N\mathcal T)*_N\mathcal T^{i+2}*_N\mathcal Y*_N\mathcal A^\pi*_N\mathcal A^i
\nonumber\\&+&\mathcal T-\mathcal A^\pi *_N\mathcal Y*_N\mathcal T^2.
\end{eqnarray*}

(b)~$\mathcal Y*_N\mathcal A^\pi*_N\mathcal S=0$, then
\begin{eqnarray*}
\mathcal S^{D}&=&-\sum_{i=0}^{k-1}\mathcal A^i*_N\mathcal A^\pi*_N\mathcal Y*_N\mathcal T^{i+2}*_N(\mathcal I-\mathcal T*_N\mathcal Y*_N\mathcal A^\pi)
\nonumber\\&+&\mathcal T-\mathcal T^2*_N\mathcal Y*_N\mathcal A^\pi.
\end{eqnarray*}
where $\ind(\mathcal A)=k$.
\end{theorem}
\begin{proof}
The results of the theorem follow directly from Lemma \ref{2.2} and Lemma \ref{2.7sae}.
\end{proof}

Note that the above equation $(a)$ also generalizes \cite[Theorem 2.5]{Daochangf}, while $(b)$ also extends \cite[Theorem 2.9]{Daochangf}. In what follows, several special cases of Theorem \ref{Thm3.3} will be discussed, providing a generalization of certain established results.

When $\mathcal D=\mathcal I$, it follows that $\mathcal D^\pi=0$. Utilizing Theorem \ref{Thm3.3}, we obtain the following corollaries.

\begin{corollary}\label{DI1}
If $\mathcal K*_N\mathcal Z^{\pi}*_N\mathcal H=0$, $\ind(\mathcal A)=k$ and if

(a)~$\mathcal A*_N\mathcal A^\pi*_N\mathcal C*_N\mathcal B=\mathcal C*_N\mathcal B*_N\mathcal A^\pi*_N\mathcal C*_N\mathcal B$, then
\begin{eqnarray*}
(\mathcal A-\mathcal C*_N\mathcal B)^{D}&=&\sum_{i=0}^{k-1}\big[\mathcal I-\mathcal A^\pi*_N\mathcal C*_N\mathcal B*_N(\mathcal A^D+\mathcal K*_N\mathcal Z^D*_N\mathcal H)\big]*_N(\mathcal A^D+\mathcal K*_N\mathcal Z^D*_N\mathcal H)^{i+2}\\&*_N&(\mathcal A-\mathcal C*_N\mathcal B)*_N\mathcal A^\pi*_N\mathcal A^i
-\mathcal A^\pi *_N\mathcal C*_N\mathcal B*_N(\mathcal A^D+\mathcal K*_N\mathcal Z^D*_N\mathcal H)^2
\nonumber\\&+&\mathcal A^D+\mathcal K*_N\mathcal Z^D*_N\mathcal H.
\end{eqnarray*}

(b)~$\mathcal C*_N\mathcal B*_N\mathcal A*_N\mathcal A^\pi=\mathcal C*_N\mathcal B*_N\mathcal A^\pi*_N\mathcal C*_N\mathcal B$, then
\begin{eqnarray*}
(\mathcal A-\mathcal C*_N\mathcal B)^{D}&=&\sum_{i=0}^{k-1}\mathcal A^\pi*_N\mathcal A^i*_N(\mathcal A-\mathcal C*_N\mathcal B)*_N(\mathcal A^D+\mathcal K*_N\mathcal Z^D*_N\mathcal H)^{i+2}\\&*_N&\big[\mathcal I-(\mathcal A^D+\mathcal K*_N\mathcal Z^D*_N\mathcal H)*_N\mathcal C*_N\mathcal B*_N\mathcal A^\pi \big]
+\mathcal A^D+\mathcal K*_N\mathcal Z^D*_N\mathcal H
\nonumber\\&-&(\mathcal A^D+\mathcal K*_N\mathcal Z^D*_N\mathcal H)^2*_N\mathcal C*_N\mathcal B*_N\mathcal A^\pi
\\&=&-\sum_{i=0}^{k-1}\mathcal A^\pi*_N\mathcal A^i*_N\mathcal C*_N\mathcal B*_N(\mathcal A^D+\mathcal K*_N\mathcal Z^D*_N\mathcal H)^{i+2}\\&*_N&\big[\mathcal I-(\mathcal A^D+\mathcal K*_N\mathcal Z^D*_N\mathcal H)*_N\mathcal C*_N\mathcal B*_N\mathcal A^\pi \big]
+\mathcal A^D+\mathcal K*_N\mathcal Z^D*_N\mathcal H
\nonumber\\&-&(\mathcal A^D+\mathcal K*_N\mathcal Z^D*_N\mathcal H)^2*_N\mathcal C*_N\mathcal B*_N\mathcal A^\pi.
\end{eqnarray*}
\end{corollary}

Notably, the preceding corollary provides a generalization of \cite[Theorem 2.1]{daochang29}. By strengthening the conditions of Corollary \ref{DI1}, we get the following results, which can also generalize \cite[Theorem 2.1]{daochang29}.

\begin{corollary}\label{DI2}
If $\mathcal C*_N\mathcal Z^{\pi}*_N\mathcal B=0$, $\ind(\mathcal A)=k$ and if

(a)~$\mathcal A^\pi*_N\mathcal C=0$, then
\begin{eqnarray*}
(\mathcal A-\mathcal C*_N\mathcal B)^{D}&=&\sum_{i=0}^{k-1}(\mathcal A^D+\mathcal K*_N\mathcal Z^D*_N\mathcal H)^{i+2}*_N(\mathcal A-\mathcal C*_N\mathcal B)*_N\mathcal A^\pi*_N\mathcal A^i
\nonumber\\&+&\mathcal A^D+\mathcal K*_N\mathcal Z^D*_N\mathcal H.
\end{eqnarray*}

(b)~$\mathcal B*_N\mathcal A^\pi=0$, then
\begin{eqnarray*}
(\mathcal A-\mathcal C*_N\mathcal B)^{D}&=&\sum_{i=0}^{k-1}\mathcal A^\pi*_N\mathcal A^i*_N(\mathcal A-\mathcal C*_N\mathcal B)*_N(\mathcal A^D+\mathcal K*_N\mathcal Z^D*_N\mathcal H)^{i+2}
\\&+&\mathcal A^D+\mathcal K*_N\mathcal Z^D*_N\mathcal H
\\&=&-\sum_{i=0}^{k-1}\mathcal A^\pi*_N\mathcal A^i*_N\mathcal C*_N\mathcal B*_N(\mathcal A^D+\mathcal K*_N\mathcal Z^D*_N\mathcal H)^{i+2}
\\&+&\mathcal A^D+\mathcal K*_N\mathcal Z^D*_N\mathcal H.
\end{eqnarray*}
\end{corollary}

\begin{corollary}\cite[Theorem 2.1]{daochang29}\label{DI3}
If $\mathcal A^{\pi}*_N\mathcal C=0$, $\mathcal B*_N\mathcal A^{\pi}=0$ and $\mathcal C*_N\mathcal Z^{\pi}*_N\mathcal B=0$, then
\begin{eqnarray*}
(\mathcal A-\mathcal C*_N\mathcal B)^{D}&=&\mathcal A^D+\mathcal K*_N\mathcal Z^D*_N\mathcal H.
\end{eqnarray*}
\end{corollary}
Note that Corollary \ref{DI2} eliminates one constraint above directly.

Suppose that $\mathcal Z$ is invertible. Then $\mathcal Z^D=\mathcal Z^{-1}$ and $\mathcal Z^\pi=0$. By Theorem \ref{Thm3.3}, the following corollaries can be derived.
\begin{corollary}\label{Z-11}
If $\mathcal K*_N\mathcal D^{\pi}*_N\mathcal Z^{-1}*_N\mathcal H=0$, $\ind(\mathcal A)=k$ and if

(a)~$\mathcal A^\pi*_N\mathcal C=0$, then
\begin{eqnarray*}
\mathcal S^{D}&=&-\sum_{i=0}^{k-1}(\mathcal A^D+\mathcal K*_N\mathcal Z^{-1}*_N\mathcal H)^{i+2}*_N\mathcal Y*_N\mathcal A^\pi*_N\mathcal A^i
\nonumber\\&+&\mathcal A^D+\mathcal K*_N\mathcal Z^{-1}*_N\mathcal H.
\end{eqnarray*}

(b)~$\mathcal B*_N\mathcal A^\pi=0$, then
\begin{eqnarray*}
\mathcal S^{D}&=&\sum_{i=0}^{k-1}\mathcal A^\pi*_N\mathcal A^i*_N\mathcal S*_N(\mathcal A^D+\mathcal K*_N\mathcal Z^{-1}*_N\mathcal H)^{i+2}
\\&+&\mathcal A^D+\mathcal K*_N\mathcal Z^{-1}*_N\mathcal H
\\&=&-\sum_{i=0}^{k-1}\mathcal A^\pi*_N\mathcal A^i*_N\mathcal Y*_N(\mathcal A^D+\mathcal K*_N\mathcal Z^{-1}*_N\mathcal H)^{i+2}
\\&+&\mathcal A^D+\mathcal K*_N\mathcal Z^{-1}*_N\mathcal H.
\end{eqnarray*}
\end{corollary}

Furthermore, the above corollary extends \cite[Corollary 2.2]{daochang10}. To present the generalization more clearly, we state the following corollary.
\begin{corollary}\label{Z-12}
If $\mathcal A^\pi*_N\mathcal C=0$, $\mathcal B*_N\mathcal A^\pi=0$ and $\mathcal C*_N\mathcal D^{\pi}*_N\mathcal Z^{-1}*_N\mathcal B=0$, then
\begin{eqnarray*}
\mathcal S^{D}&=&\mathcal A^D+\mathcal K*_N\mathcal Z^{-1}*_N\mathcal H.
\end{eqnarray*}
\end{corollary}

It is noteworthy that Corollary \ref{Z-12} directly eliminates one of the conditions present in \cite[Corollary 2.2]{daochang10}.

Analogous to Lemma \ref{2.2}, the following results about the Drazin inverse of $\mathcal Z$ are established.

\begin{corollary}\label{d2.1}
If $\mathcal Z*_N\mathcal D^\pi$ is u-nilpotent and if

(a)~$\mathcal Z*_N\mathcal D^{\pi}*_N\mathcal R*_N\mathcal D^e=0$, then
\begin{eqnarray*}
\mathcal Z^{D}&=&\sum_{i=0}^{u-2}\left[(\mathcal Z_{\mathcal D^e}^{D})^{i+2}
-\mathcal D^{\pi}*_N\mathcal R*_N\mathcal (Z_{\mathcal D^e}^{D})^{i+3}\right]*_N\mathcal Z*_N\mathcal D^\pi *_N\mathcal Z^i
\nonumber\\&+&\mathcal Z_{\mathcal D^e}^{D}
-\mathcal D^\pi *_N\mathcal R*_N(\mathcal Z_{\mathcal D^e}^{D})^2,
\end{eqnarray*}

(b)~$\mathcal Z*_N\mathcal D^{e}*_N\mathcal R*_N\mathcal D^\pi=0$, then
\begin{eqnarray*}
\mathcal Z^{D}&=&\sum_{i=0}^{u-2}\mathcal Z^{i+1}*_N\mathcal D^\pi
*_N\left[(\mathcal Z_{\mathcal D^e}^{D})^{i+2}
-\mathcal D^{\pi}*_N\mathcal R*_N\mathcal (Z_{\mathcal D^e}^{D})^{i+3}\right]
\nonumber\\&+&\mathcal Z_{\mathcal D^e}^{D}
-\mathcal D^\pi *_N\mathcal R*_N(\mathcal Z_{\mathcal D^e}^{D})^2,
\end{eqnarray*}
\end{corollary}

Clearly, we observe that Corollary \ref{d2.1} generalizes \cite[Lemma 2.9$(a)$]{Dijanaf} and \cite[Lemma 2.9$(b)$]{Dijanaf}, respectively.
\begin{corollary}\label{d2.2}
If $\mathcal D^\pi*_N\mathcal Z$ is v-nilpotent and if

(a)~$\mathcal D^e*_N\mathcal R*_N\mathcal D^{\pi}*_N\mathcal Z=0$, then
\begin{eqnarray*}
\mathcal Z^{D}&=&\sum_{i=0}^{v-2}\mathcal Z^i*_N\mathcal D^\pi*_N\mathcal Z*_N\left[(\mathcal Z_{\mathcal D^e}^{D})^{i+2}
-(\mathcal Z_{\mathcal D^e}^{D})^{i+3}*_N\mathcal R*_N\mathcal D^\pi\right]
\nonumber\\&+&\mathcal Z_{\mathcal D^e}^{D}
-(\mathcal Z_{\mathcal D^e}^{D})^2*_N\mathcal R*_N\mathcal D^\pi,
\end{eqnarray*}

(b)~$\mathcal D^\pi*_N\mathcal R*_N\mathcal D^{e}*_N\mathcal Z=0$, then
\begin{eqnarray*}
\mathcal Z^{D}&=&\sum_{i=0}^{v-2}
\left[(\mathcal Z_{\mathcal D^e}^{D})^{i+2}
-(\mathcal Z_{\mathcal D^e}^{D})^{i+3}*_N\mathcal R*_N\mathcal D^\pi\right]*_N\mathcal D^\pi*_N\mathcal Z^{i+1}
\nonumber\\&+&\mathcal Z_{\mathcal D^e}^{D}
-(\mathcal Z_{\mathcal D^e}^{D})^2*_N\mathcal R*_N\mathcal D^\pi,
\end{eqnarray*}
\end{corollary}

Meanwhile, notice that Corollary \ref{d2.2} extends \cite[Lemma 2.9$(c)$]{Dijanaf} and \cite[Lemma 2.9$(d)$]{Dijanaf}, respectively. Furthermore, we refine and extend  \cite[Lemma 2.9 $(a)(c)$]{Dijanaf} by removing one of its conditions and proposing the following corollaries.
\begin{corollary}\label{d2.3}
If $\mathcal Z*_N\mathcal D^\pi*_N\mathcal R=0$ and $\ind(\mathcal D)=l$, then
\begin{eqnarray*}
\mathcal Z^{D}&=&\sum_{i=0}^{l-1}\left[(\mathcal Z_{\mathcal D^e}^{D})^{i+2}
-\mathcal D^{\pi}*_N\mathcal R*_N\mathcal (Z_{\mathcal D^e}^{D})^{i+3}\right]*_N\mathcal Z*_N\mathcal D^\pi *_N\mathcal D^i
\nonumber\\&+&\mathcal Z_{\mathcal D^e}^{D}
-\mathcal D^\pi *_N\mathcal R*_N(\mathcal Z_{\mathcal D^e}^{D})^2.
\end{eqnarray*}
\end{corollary}

\begin{corollary}\label{d2.4}
If $\mathcal R*_N\mathcal D^\pi*_N\mathcal Z=0$ and $\ind(\mathcal D)=l$, then
\begin{eqnarray*}
\mathcal Z^{D}&=&\sum_{i=0}^{l-1}D^i*_N\mathcal D^\pi*_N\mathcal Z*_N\left[(\mathcal Z_{\mathcal D^e}^{D})^{i+2}
-(\mathcal Z_{\mathcal D^e}^{D})^{i+3}*_N\mathcal R*_N\mathcal D^\pi\right]
\nonumber\\&+&\mathcal Z_{\mathcal D^e}^{D}
-(\mathcal Z_{\mathcal D^e}^{D})^2*_N\mathcal R*_N\mathcal D^\pi.
\end{eqnarray*}
\end{corollary}

\section{Numerical example}
We now present the following example to illustrate the main theorem.
\begin{example}{\label{exp0.1}}
    Let $\mathcal{A}\in\mathbb{R}^{(2\times 3)\times (2\times 3)}$  be a tensor defined as follows:
\[\mathcal{A}(:,:,1,1) = \begin{bmatrix} \ 1 & 0 & 0\ \\ \ 0 & 0 & 0\ \end{bmatrix}, \quad \mathcal{A}(:,:,2,1) = \begin{bmatrix} \ 0 & 0 & 0\ \\ \ 1 & 0 & 0\ \end{bmatrix}, \quad \mathcal{A}(:,:,1,2) = \begin{bmatrix} \ 0 & 1 & 0\ \\ \ 0 & 0 & 0\ \end{bmatrix},\]
     \[\mathcal{A}(:,:,2,2) = \begin{bmatrix} \ 0 & 0 & 0\ \\ \ 0 & 1 & 0\ \end{bmatrix}, \quad \mathcal{A}(:,:,1,3) = \begin{bmatrix} \ 0 & 0 & 1\ \\ \ 1 & 0 & 0\ \end{bmatrix}, \quad \mathcal{A}(:,:,2,3) = \begin{bmatrix} \ 0 & 0 & 0\ \\ \ 0 & 0 & 0\ \end{bmatrix},\]

     $\mathcal{B}\in\mathbb{R}^{(3\times 2)\times (2\times 3)}$  with entries given by
\[\mathcal{B}(:,:,1,1) = \begin{bmatrix} \ 0 & 0\ \\ \ 0 & 0\ \\ \ 0 & 0\ \end{bmatrix}, \quad \mathcal{B}(:,:,2,1) = \begin{bmatrix} \ 0 & 0\ \\ \ 1/4 & 0\ \\ \ 0 & 0\ \end{bmatrix}, \quad \mathcal{B}(:,:,1,2) = \begin{bmatrix} \ 1/4 & 0\ \\ \ 0 & 0\ \\ \ 0 & 0\ \end{bmatrix},\]
     \[\mathcal{B}(:,:,2,2) = \begin{bmatrix} \ 0 & 0\ \\ \ 0 & 1/4\ \\ \ 0 & 0\ \end{bmatrix}, \quad \mathcal{B}(:,:,1,3) = \begin{bmatrix} \ 0 & 0\ \\ \ 0 & 0\ \\ \ 0 & 0\ \end{bmatrix}, \quad \mathcal{B}(:,:,2,3) = \begin{bmatrix} \ 0 & 0\ \\ \ 0 & 0\ \\ \ 0 & 0\ \end{bmatrix},\]

       $\mathcal{C}\in\mathbb{R}^{(2\times 3)\times (3\times 2)}$ whose entries are
\[\mathcal{C}(:,:,1,1) = \begin{bmatrix} \ 1/4 & 0 & 1/4\ \\ \ 0 & 0 & 0\ \end{bmatrix}, \quad \mathcal{C}(:,:,2,1) = \begin{bmatrix} \ 0 & 1/4 & 0\ \\ \ 0 & 0 & 0\ \end{bmatrix}, \quad \mathcal{C}(:,:,3,1) = \begin{bmatrix} \ 0 & 0 & 0\ \\ \ 0 & 0 & 0\ \end{bmatrix},\]
     \[\mathcal{C}(:,:,1,2) = \begin{bmatrix} \ 0 & 0 & 0\ \\ \ 0 & 1/4 & 0\ \end{bmatrix}, \quad \mathcal{C}(:,:,2,2) = \begin{bmatrix} \ 0 & 0 & 0\ \\ \ 1/4 & 0 & 1/4\ \end{bmatrix}, \quad \mathcal{C}(:,:,3,2) = \begin{bmatrix} \ 1/4 & 1/4 & 1/4\ \\ \ 1/4 & 0 & 1/4\ \end{bmatrix},\]

   and $\mathcal{D}\in\mathbb{R}^{(3\times 2)\times (3\times 2)}$ be defined by
 \[\mathcal{D}(:,:,1,1) = \begin{bmatrix} \ 1/2 & 0\ \\ \ 0 & 0\ \\ \ 0 & 0\ \end{bmatrix}, \quad \mathcal{D}(:,:,2,1) = \begin{bmatrix} \ 0 & 0\ \\ \ 1/2 & 0\ \\ \ 0 & 0\ \end{bmatrix}, \quad \mathcal{D}(:,:,3,1) = \begin{bmatrix} \ 0 & 0\ \\ \ 0 & 0\ \\ \ 1/2 & 0\ \end{bmatrix},\]
     \[\mathcal{D}(:,:,1,2) = \begin{bmatrix} \ 0 & 1/2\ \\ \ 1/2 & 0\ \\ \ 0 & 0\ \end{bmatrix}, \quad \mathcal{D}(:,:,2,2) = \begin{bmatrix} \ 0 & 0\ \\ \ 0 & 1/2\ \\ \ 0 & 1/2\ \end{bmatrix}, \quad \mathcal{D}(:,:,3,2) = \begin{bmatrix} \ 0 & 0\ \\ \ 1/2 & 0\ \\ \ 0 & 1/2\ \end{bmatrix}.\]

The following results can be verified through straightforward computation.
\[\mathcal{A}^D(:,:,1,1) = \begin{bmatrix} \ 1 & 0 & 0\ \\ \ 0 & 0 & 0\ \end{bmatrix}, \quad \mathcal{A}^D(:,:,2,1) = \begin{bmatrix} \ 0 & 0 & 0\ \\ \ 1 & 0 & 0\ \end{bmatrix}, \quad \mathcal{A}^D(:,:,1,2) = \begin{bmatrix} \ 0 & 1 & 0\ \\ \ 0 & 0 & 0\ \end{bmatrix},\]
     \[\mathcal{A}^D(:,:,2,2) = \begin{bmatrix} \ 0 & 0 & 0\ \\ \ 0 & 1 & 0\ \end{bmatrix}, \quad \mathcal{A}^D(:,:,1,3) = \begin{bmatrix} \ 0 & 0 & 1\ \\ \ -1 & 0 & 0\ \end{bmatrix}, \quad \mathcal{A}^D(:,:,2,3) = \begin{bmatrix} \ 0 & 0 & 0\ \\ \ 0 & 0 & 0\ \end{bmatrix},\]
\[\mathcal{D}^D(:,:,1,1) = \begin{bmatrix} \ 2 & 0\ \\ \ 0 & 0\ \\ \ 0 & 0\ \end{bmatrix}, \quad \mathcal{D}^D(:,:,2,1) = \begin{bmatrix} \ 0 & 0\ \\ \ 2 & 0\ \\ \ 0 & 0\ \end{bmatrix}, \quad \mathcal{D}^D(:,:,3,1) = \begin{bmatrix} \ 0 & 0\ \\ \ 0 & 0\ \\ \ 2 & 0\ \end{bmatrix},\]
     \[\mathcal{D}^D(:,:,1,2) = \begin{bmatrix} \ 0 & 2\ \\ \ -2 & 0\ \\ \ 0 & 0\ \end{bmatrix}, \quad \mathcal{D}^D(:,:,2,2) = \begin{bmatrix} \ 0 & 0\ \\ \ 2 & 2\ \\ \ 0 & -2\ \end{bmatrix}, \quad \mathcal{D}^D(:,:,3,2) = \begin{bmatrix} \ 0 & 0\ \\ \ 2 & 0\ \\ \ 0 & 2\ \end{bmatrix},\] and
\[\mathcal{Z}^D(:,:,1,1) = \begin{bmatrix} \ 128/65 & 0\ \\ \ -16/65 & 0\ \\ \ 0 & 0\ \end{bmatrix}, \quad \mathcal{Z}^D(:,:,2,1) = \begin{bmatrix} \ 16/65 & 0\ \\ \ 128/65 & 0\ \\ \ 0 & 0\ \end{bmatrix}, \quad \mathcal{Z}^D(:,:,3,1) = \begin{bmatrix} \ 0 & 0\ \\ \ 0 & 0\ \\ \ 2 & 0\ \end{bmatrix},\]
     \[\mathcal{Z}^D(:,:,1,2) = \begin{bmatrix} \ -63/260 & 2\ \\ \ -439/260 & 1/4\ \\ \ 0 & -1/4\ \end{bmatrix}, \quad \mathcal{Z}^D(:,:,2,2) = \begin{bmatrix} \ 2/65 & 0\ \\ \ 146/65 & 2\ \\ \ 0 & -2\ \end{bmatrix}, \quad \mathcal{Z}^D(:,:,3,2) = \begin{bmatrix} \ 0 & 0\ \\ \ -2 & 0\ \\ \ 0 & 2\ \end{bmatrix}.\]

It can be readily demonstrated that the above results satisfy
 $\mathcal K*_2\mathcal D^{\pi}*_2\mathcal Z^{D}*_2\mathcal H
   =\mathcal K*_2\mathcal D^{D}*_2\mathcal Z^{\pi}*_2\mathcal H$,
   $\ind(\mathcal A)=1$ and
$\mathcal S*_2\mathcal A^\pi*_2\mathcal Y=0$.
Then, by Theorem \ref{Thm3.3} $(a)$, we obtain

\[\mathcal{S}^D(:,:,1,1) = \begin{bmatrix} \ 1 & 0 & 0\ \\ \ 0 & 0 & 0\ \end{bmatrix}, \quad \mathcal{S}^D(:,:,2,1) = \begin{bmatrix} \ 1/65 & 8/65 & 1/65\ \\ \ 64/65 & 0 & 0\ \end{bmatrix}, \]
\[\mathcal{S}^D(:,:,1,2) = \begin{bmatrix} \ 8/65 & 64/65 & 8/65\ \\ \ -8/65 & 0 & 0\ \end{bmatrix},
\quad     \mathcal{S}^D(:,:,2,2) = \begin{bmatrix} \ -8/65 & 1/65 & -8/65\ \\ \ 8/65 & 1 & 0\ \end{bmatrix},\]
\[  \mathcal{S}^D(:,:,1,3) = \begin{bmatrix} \ -1/65 & -8/65 & 64/65\ \\ \ -64/65 & 0 & 0\ \end{bmatrix}, \quad \mathcal{S}^D(:,:,2,3) = \begin{bmatrix} \ 0 & 0 & 0\ \\ \ 0 & 0 & 0\ \end{bmatrix}.\]

Next, we regard $-\mathcal{Y}$ as a perturbation of $\mathcal{A}$.
It satisfies $-\mathcal{Y}=-\mathcal{A}*_2\mathcal{A}^D*_2\mathcal{Y}*_2\mathcal{A}*_2\mathcal{A}^D$ and $\Vert \mathcal{A}^D*_2\mathcal{Y}  \Vert_2=253/765<1$.
Hence, we can verify that
$\mathcal{S}^D-\mathcal{A}^D=-\mathcal{S}^D*_2\mathcal{Y}*_2\mathcal{A}^D=-\mathcal{A}^D*_2\mathcal{Y}*_2\mathcal{S}^D$. Then, we obtain
$$\Vert \mathcal{S}^D-\mathcal{A}^D \Vert_2=341/972, \quad
\Vert \mathcal{A}^D\Vert_2=2158/881,$$
and thus, the perturbation error is given by
$$\frac{\Vert \mathcal{S}^D-\mathcal{A}^D \Vert_2}{\Vert \mathcal{A}^D\Vert_2}\leq  \frac{\Vert \mathcal{A}^D*_2\mathcal{Y}  \Vert_2}{1-\Vert \mathcal{A}^D*_2\mathcal{Y}  \Vert_2}.$$

Furthermore, we modify the definition of $\mathcal{D}$ as follows:
 \[\mathcal{D}(:,:,1,1) = \begin{bmatrix} \ 1/a & 0\ \\ \ 0 & 0\ \\ \ 0 & 0\ \end{bmatrix}, \quad \mathcal{D}(:,:,2,1) = \begin{bmatrix} \ 0 & 0\ \\ \ 1/a & 0\ \\ \ 0 & 0\ \end{bmatrix}, \quad \mathcal{D}(:,:,3,1) = \begin{bmatrix} \ 0 & 0\ \\ \ 0 & 0\ \\ \ 1/a & 0\ \end{bmatrix},\]
     \[\mathcal{D}(:,:,1,2) = \begin{bmatrix} \ 0 & 1/a\ \\ \ 1/a & 0\ \\ \ 0 & 0\ \end{bmatrix}, \quad \mathcal{D}(:,:,2,2) = \begin{bmatrix} \ 0 & 0\ \\ \ 0 & 1/a\ \\ \ 0 & 1/a\ \end{bmatrix}, \quad \mathcal{D}(:,:,3,2) = \begin{bmatrix} \ 0 & 0\ \\ \ 1/a & 0\ \\ \ 0 & 1/a\ \end{bmatrix}.\]

Let $a=c\cdot \varepsilon$, where $c$ is a random number in the interval $[-1,1]$, and $\varepsilon$ is a given coefficient.
Although the Drazin inverse of $\mathcal S$ can be calculated either by Mathematica's \texttt{DrazinInverse} function or by Theorem~\ref{Thm3.3} $(a)$, the resulting numerical errors differ. Therefore, let $\mathcal X_1$ denote the Drazin inverse obtained via Theorem~\ref{Thm3.3} $(a)$, and $\mathcal X_2$ is the one obtained via the \texttt{DrazinInverse} function. Accordingly, the following residual errors are considered:
\begin{center}
$r_1=\Vert  \mathcal X*_2\mathcal S*_2\mathcal X- \mathcal X\Vert$, \
$r_2=\Vert \mathcal S*_2\mathcal X- \mathcal X*_2\mathcal S \Vert$ \
and \  $r_3=\Vert \mathcal S^k*_2\mathcal X*_2\mathcal S-\mathcal S^k \Vert$,
\end{center}
where $k=\ind(\mathcal S)$. The specific results are shown in Table \ref{tab0.1}.
\begin{table}[H]
\caption{\textbf{Example \ref{exp0.1}} by Theorem \ref{Thm3.3} $(a)$}
\label{tab0.1}
\centering
\begin{tabular}{ccccc}
\toprule
 &  $\varepsilon$  & $r_1$ & $r_2$ & $r_3$  \\
\midrule
$\mathcal X_1$ & \multirow{2}{*}{$10^{}$} & 1.1365e-16 & 6.6203e-17 & 9.5755e-17 \\
$\mathcal X_2$ & &  1.4677e-15&       5.2324e-16                                              & 1.3598e-15  \\
\midrule
$\mathcal X_1$ & \multirow{2}{*}{$10^{-1}$} &1.0925e-17& 5.5841e-18 & 9.3707e-18\\
$\mathcal  X_2$ & &2.4633e-15  &   5.8333e-16                                                  & 2.4888e-15  \\
\midrule
$\mathcal X_1$ & \multirow{2}{*}{$10^{-3}$} &1.0466e-19 & 6.5919e-20 & 9.7709e-20 \\
$\mathcal  X_2$ & & 4.4855e-15 &        7.9477e-16                                             &  4.2230e-15  \\
\midrule
$\mathcal  X_1$ & \multirow{2}{*}{$10^{-5}$} &3.1536e-22 & 1.5886e-22 &3.2852e-22  \\
$\mathcal X_2$ & & 1.0827e-15 &             4.2163e-16                                        &  1.5438e-15  \\
\bottomrule
\end{tabular}
\end{table}
By comparison, the error associated with $\mathcal X_1$ is always smaller than that of $\mathcal X_2$, regardless of the magnitude of $\varepsilon$.
Moreover, as $\varepsilon$ decreases, the calculation error of $\mathcal X_1$ also decreases. These results indicate that Theorem~\ref{Thm3.3} $(a)$  provides a more accurate method for computing the Drazin inverse.
\end{example}

%\section*{References}
%\bibliography{mybibfile}

\end{document}